\documentclass{article}
\usepackage[english]{babel}

\usepackage{amsmath,amssymb,setspace}
\usepackage{amsthm}
\usepackage[english]{babel}
\usepackage{amsfonts}
\usepackage{calligra}
\usepackage[T1]{fontenc}
\usepackage{xstring}

\usepackage[colorlinks,citecolor=blue,urlcolor=blue,bookmarks=false,hypertexnames=true]{hyperref}
\hypersetup{
    colorlinks=true,
    linkcolor=blue,
    filecolor=magenta,
    urlcolor=cyan,
}
\usepackage{graphics}
\usepackage[numbers,sort&compress]{natbib}
\usepackage{subcaption}
\usepackage{multicol}
\date{}
\usepackage{longtable}
\usepackage{float}
\usepackage[margin=0.7 in]{geometry}
\usepackage{graphicx}
\usepackage{mathtools}

\newtheorem{definition}{Definition}
\theoremstyle{plain}
\theoremstyle{definition}
\theoremstyle{remark}
\newtheorem{theorem}{Theorem}
\newtheorem{corollary}{Corollary}

\newtheorem{remark}{Remark}

\newtheorem{example}{Example}

\title{The Mellin Transform and Non-local Derivatives of Fractal Calculus}
\author{Alireza Khalili Golmankhaneh \footnote{Corresponding author} \\
Department of Physics, Urmia Branch, Islamic Azad University, Urmia 63896,  Iran\\
	Kerri Welch\\
Faculty at California Institute of Integral Studies, San Francisco, CA, U.S.A.\\
Cristina Serpa\\
Instituto Superior de Engenharia de Lisboa (ISEL), \\Instituto Polit\'{e}cnico de Lisboa, Lisbon, Portugal\\
Centro de Matem\'{a}tica, Aplica\c{c}\~{o}es Fundamentais e Investiga\c{c}\~{a}o Operacional (CMAFcIO),\\ Faculdade de Ci\^{e}ncias, Universidade de Lisboa, Lisbon, Portugal\\
Palle E. T. J{\o}rgensen\\
Department of Mathematics,The University of Iowa, Iowa City, IA
52242-1419, U.S.A.
}

\date{\today}
\begin{document}

\maketitle

\begin{abstract}
In this paper, the fractal calculus of fractal sets and fractal curves are compared. The analogues of the Riemann-Liouville and the  Caputo  integrals and derivatives are defined for the fractal curves which are non-local derivatives. The analogous for the fractional Laplace concepts are defined to solve fractal non-local differential equations on fractal curves. The fractal local Mellin and fractal non-local transforms are defined to solve fractal differential equations. We present tables and examples to illustrate the results.
\end{abstract}

\section{Introduction}
Fractal geometry was introduced by Benoit Mandelbrot \cite{b-1,b-6}  which involves complex geometric shapes where their fractal dimensions exceed their topological dimensions \cite{falconer1999techniques}. The fractals are self-similar and often have non-integer dimensions \cite{b-2}. Since fractals have different measures, such as the Hausdorff measure, length, surface, and volume,  the measures of Euclidean geometric shapes, cannot be used in fractal analysis \cite{ma-12,khelifi2020}. \\
The fractal analysis was formulated by many researchers using harmonic analysis \cite{ma-8,ma-13}, measure theory \cite{freiberg2002harmonic}, probabilistic method \cite{ma-7}, fractional space \cite{stillinger1977axiomatic}, fractional calculus \cite{ma-6}, and non-standard methods \cite{Nottale-4}. The development of fractal calculus benefits also from new techniques to obtain fractal functions from real data, e.g., the fractal regression gives applied scientists the fractal shape of data \cite{CRISTINA4}. The conjugation of theoretical studies with fitting approximative functions is a strong means to better understand some concrete problems of applied science that do not fit the classical geometric setting.\\
Fractal calculus is formulated by a generalization of ordinary calculus which involves differential equations. Their solutions are functions with fractal support such as fractal sets and curves \cite{parvate2009calculus,parvate2011calculus,AD-23,BookAlireza}. Fractal calculus is algorithmic and simpler  compared with other methods \cite{BookAlireza}. The fractal calculus approach with the fractional spaces method is compared in applications such as Einstein's field equations \cite{el2021fractionalu}. Fractal stochastic differential equations are defined and the fractional Brownian motion and diffusion in the medium with fractal structures are categorized  \cite{golmankhaneh2021equilibrium,khalili2019fractalcat,khalili2019random,banchuin20224noise,banchuin2022noise,khalili2019fractal,golmankhaneh2021fractalBro,golmankhaneh2018sub}.
The fractal differential equations were solved by different methods and their stability conditions are given \cite{golmankhaneh2019sumudu,khalili2021hyers,cetinkaya2021general,golmankhaneh2021local}.
The fractal calculus was generalized to involve Cantor cubes and Cantor tartan \cite{golmankhaneh2018fractalt} and defined the Laplace equation \cite{khalili2021laplace}. The derivative and integral of the Weierstrass function are given \cite{gowrisankar2021fractal}.\\
The layout of the manuscript is as follows:\\
In Section \ref{1g}, we compare and review the fractal calculus of fractal sets and curves. We define non-local fractal derivatives and integrals such as analogues of the Riemann-Liouville derivatives and Caputo fractional derivatives in Section \ref{2g}. The analogues Laplace transform of Riemann-Liouville derivatives and Caputo  derivatives are defined to solve non-local differential equations on fractals in Section \ref{3g}. In Section \ref{40g}, the analogue of the Mellin transform  is defined and used to solve a fractal differential equation. We devoted Section \ref{5g} to a conclusion.

\section{Basic tools \label{1g}}
In this section, we summarize  the fractal calculus on fractal curves and fractal sets \cite{parvate2009calculus,parvate2011calculus,AD-23,BookAlireza}.
\subsection{Fractal calculus on fractal curves}
\begin{definition}
Consider a fractal curve $\mathfrak{F}$ and a subdivision $\mathfrak{P}_{[c_{1},c_{2}]}, [c_{1},c_{2}]\in [a_{0},b_{0}] \subset\mathbb{R} $. Then the mass function is defined by
\begin{equation}
\zeta^{\alpha}(\mathfrak{F},c_{1},c_{2})=\lim_{\delta\rightarrow0} \inf_{|P|\leq \delta}\sum_{i=0}^{n-1}
\frac{|\mathbf{v}(z_{i+1})-\mathbf{v}(z_{i})|^{\alpha}}{\Gamma(\alpha+1)},
\end{equation}
where $|.|$ is the Euclidean norm of $\mathbb{R}^{n}$, $1\leq \alpha\leq n$, $P_{[c_{1},c_{2}]}=\{c_{1}=z_{0},...,z_{n}=c_{2}\}$, $|P|=\max_{0\leq i\leq n-1}(t_{i+1}-t_{i})$ for a subdivision $\mathfrak{P}$ and $\Gamma(*)$ is the gamma function.
\end{definition}
\begin{definition}
The dimension of $\mathfrak{F}$  is defined by
\begin{align}
  \dim_{\zeta}(\mathfrak{F})&=\inf\{\alpha:\zeta^{\alpha}(\mathfrak{F},c_{1},c_{2})=0\}\nonumber\\&=
\sup\{\alpha:\zeta^{\alpha}(\mathfrak{F},c_{1},c_{2})=\infty\}.
\end{align}
\end{definition}
\begin{definition}
The rise function of  $\mathfrak{F}$ is defined by
\begin{equation}
  S_{\mathfrak{F}}^{\alpha}(z)=\left\{
                      \begin{array}{ll}
                        \zeta^{\alpha}(\mathfrak{F},q_{0},z), & z\geq q_{0} ; \\
                        -\zeta^{\alpha}(\mathfrak{F},z,q_{0}), & z<q_{0},
                      \end{array}
                    \right.
\end{equation}
where $z\in [a_{0},b_{0}]$, and $S_{\mathfrak{F}}^{\alpha}(z)$  gives the mass of $\mathfrak{F}$ upto point $z$.
\end{definition}

\begin{definition}
 The $\mathfrak{F}^{\alpha}$-derivative is defined by
\begin{equation}
  D_{\mathfrak{F}}^{\alpha}f(\theta)=\underset{ \theta'\rightarrow \theta}{\mathfrak{F}_{-}lim}~
\frac{f(\theta')-f(\theta)}{J(\theta')-J(\theta)},
\end{equation}
where $\mathfrak{F}_{-}lim$ indicates the fractal limit (see in \cite{parvate2011calculus}), and  $\mathbf{v}(z)=\theta$ and $S_{\mathfrak{F}}^{\alpha}(z)=J(\theta)$.
\end{definition}
\begin{remark}
We note that the Euclidean distance from the origin upto a point $\theta=\mathbf{v}(z)$ is given by $\mathfrak{L}(\theta)=\mathfrak{L}(\mathbf{v}(z))=|\mathbf{v}(z)|.$
\end{remark}
\begin{definition}
The $\mathfrak{F}^{\alpha}$-integral is defined by
\begin{align}
  \int_{D(c_{1},c_{2})}f(\theta)d_{\mathfrak{F}}^{\alpha}\theta&=
\sup_{\mathfrak{P}[c_{1},c_{2}]}\sum_{i=0}^{n-1}
\inf_{\theta\in C(z_{i},z_{i+1})}f(\theta)(J(\theta_{i+1})-J(\theta_{i}))\nonumber\\&=
\inf_{\mathfrak{P}[c_{1},c_{2}]}\sum_{i=0}^{n-1}
\sup_{\theta\in C(z_{i},z_{i+1})}f(\theta)(J(\theta_{i+1})-J(\theta_{i})),
\end{align}
where $z_{i}=\mathbf{v}^{-1}(\theta_{i})$, and $D(c_{1},c_{2})$ is the section of  between points $\mathbf{v}(c_{1})$ and $\mathbf{v}(c_{2})$ on the fractal curve $\mathfrak{F}$ \cite{parvate2011calculus}.
\end{definition}

\subsection{Fractal calculus for fractal sets}
Here, we give a  summary  of fractal calculus of  sets $\mathbf{F}\subset [c_{1},c_{2}]\subset \mathbb{R}$ \cite{parvate2009calculus}.
\begin{definition}
The flag function of $\mathbf{F}$  is defined by
\begin{equation}
  \rho(\mathbf{F},I)=\left\{
              \begin{array}{ll}
                1, & \textmd{if}~~ \mathbf{F}\cap I\neq\emptyset;\\
                0, & \textmd{otherwise},
              \end{array}
            \right.
\end{equation}
where $I=[c_{1},c_{2}]\subset \mathbb{R}$
\end{definition}
\begin{definition}
The coarse-grained mass of $\mathbf{F}\cap [c_{1},c_{2}]$  is defined by
\begin{equation}
  \xi_{\delta}^{\alpha}(\mathbf{F},c_{1},c_{2})=\inf_{|\mathfrak{P}|\leq \delta}\sum_{i=0}^{n-1}\Gamma(\alpha+1)(z_{i+1}-z_{i})^{\alpha}
\rho(\mathbf{F},[z_{i},z_{i+1}]),
\end{equation}
where
\begin{equation}
  |\mathfrak{P}|=\max_{0\leq i\leq n-1}(z_{i+1}-z_{i}),
\end{equation}
and $0< \alpha\leq1$.
\end{definition}
\begin{definition}
The mass function of $\mathbf{F}$ is defined by
\begin{equation}
  \xi^{\alpha}(\mathbf{F},c_{1},c_{2})=
\lim_{\delta\rightarrow0}\xi_{\delta}^{\alpha}(\mathbf{F},c_{1},c_{2}).
\end{equation}
\end{definition}
\begin{definition}
The fractal dimension of $\mathbf{F}\cap [c_{1},c_{2}]$, is defined by
\begin{align}
  \dim_{\xi}(\mathbf{F}\cap [c_{1},c_{2}])&=\inf\{\alpha:\xi^{\alpha}(\mathbf{F},c_{1},c_{2})=0\}\nonumber\\&
=\sup\{\alpha:\xi^{\alpha}(\mathbf{F},c_{1},c_{2})=\infty\}.
\end{align}
\end{definition}
\begin{definition}
The integral staircase function of  $\mathbf{F}$ is defined by
\begin{equation}
 S_{\mathbf{F}}^{\alpha}(z)=\left\{
                     \begin{array}{ll}
                       \xi^{\alpha}(\mathbf{F},c_{0},z), & if~z\geq c_{0} ; \\
                      - \xi^{\alpha}(\mathbf{F},z,c_{0}), & otherwise,
                     \end{array}
                   \right.
\end{equation}
where $c_{0}\in \mathbb{R}$ is a fixed number.
\end{definition}
\begin{definition} For a function $g$ on an $\alpha$-perfect fractal set,
 $F^{\alpha}$-derivative of $g$ at $x$ in  is defined by
\begin{equation}
  D_{\mathbf{F}}^{\alpha}g(z)=\left\{
                       \begin{array}{ll}
                         \underset{ y\rightarrow z}{\mathbf{F}_{-}lim}~\frac{g(y)-g(z)}
{S_{\mathbf{F}}^{\alpha}(y)-S_{\mathbf{F}}^{\alpha}(z)}, & if~ z\in \mathbf{F}; \\
                         0, & otherwise,
                       \end{array}
                     \right.
\end{equation}
if the fractal limit $\mathbf{F}_{-}lim$ exists \cite{parvate2009calculus}.

\end{definition}

\begin{definition}
The $\mathbf{F}^{\alpha}$-integral of a bounded function $g(z), g\in B(F)$ (i.e., $g$ is a bounded function of $F$),  is defined by
\begin{align}
  \int_{a}^{b}g(z)d_{\mathbf{F}}^{\alpha}z&=\sup_{\mathfrak{P}_{[c_{1},c_{2}]}}
\sum_{i=0}^{n-1}\inf_{z\in \mathbf{F}\cap I}g(z)(S_{\mathbf{F}}^{\alpha}(z_{i+1})-S_{F}^{\alpha}(z_{i}))
\nonumber\\&=\inf_{\mathfrak{P}_{[c_{1},c_{2}]}}
\sum_{i=0}^{n-1}\sup_{z\in \mathbf{F}\cap I}g(z)(S_{\mathbf{F}}^{\alpha}(z_{i+1})-S_{\mathbf{F}}^{\alpha}(z_{i})),
\end{align}
where $z\in \mathbf{F}$, and infimum or supremum are taking on the all subdivision $\mathfrak{P}_{[c_{1},c_{2}]}$.
\end{definition}
\begin{example}
  Consider the fractal discretional equation
\begin{equation}\label{Rewww}
 D_{G}^{\alpha}y=2y-4,~~~y(0)=5,
\end{equation}
where $G$ is fractal set or fractal curve. Utilizing the conjugacy of fractal calculus with ordinary calculus \cite{parvate2009calculus,AD-23,BookAlireza} the solution of Eq.\eqref{Rewww} on a fractal set is;
\begin{equation}\label{xzdsae744}
  y(x)=2+3\exp(2 S_{\mathbf{F}}^{\alpha}(x)),
\end{equation}
and the solution of Eq.\eqref{Rewww} on a fractal curve is;
\begin{equation}\label{yhtfrd}
  y(\theta)=2+3\exp(2 S_{\mathfrak{F}}^{\alpha}(\theta)),
\end{equation}
\end{example}
In Figures \ref{sh} and \ref{h}, we have plotted  Eqs.\eqref{xzdsae744} and  \eqref{yhtfrd}, respectively.
\begin{figure}
  \centering
  \includegraphics[scale=0.5]{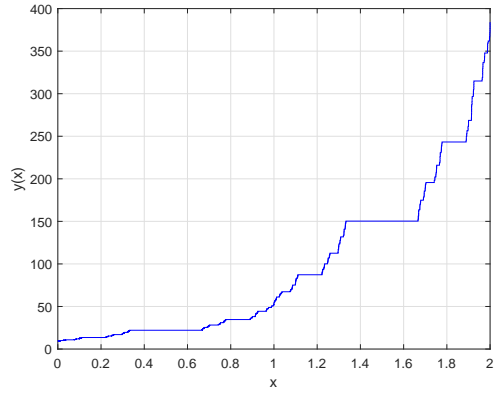}
  \caption{The graph of Eq.\eqref{xzdsae744}}\label{sh}
\end{figure}

\begin{figure}
  \centering
  \includegraphics[scale=0.5]{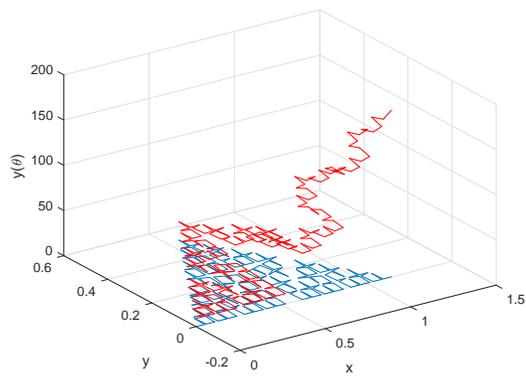}
  \caption{The graph of Eq.\eqref{yhtfrd}}\label{h}
\end{figure}

\section{Non-local derivatives on fractals \label{2g}}
In the following, the non-local integrals and derivatives on fractal spaces  are defined \cite{golmankhaneh2016non}.
\subsection{Non-local derivatives of fractal sets}
\begin{definition}
The right-sided Riemann-Liouville $F^{\alpha}$-integral of order $\beta\in \mathbb{R}$ of $g$ (function on a fractal set) is defined by
\begin{equation}
  _{z}\mathcal{I}_{b}^{\beta}g(z)=\frac{1}{\Gamma(\beta)}\int_{z}^{b}
\frac{g(\tau)}{(S_{\mathbf{F}}^{\alpha}(\tau)-
S_{\mathbf{F}}^{\alpha}(z))^{1-\beta}}
d_{\mathbf{F}}^{\alpha}\tau,
\end{equation}
and  the left-sided Riemann-Liouville $F^{\alpha}$-integral of $g$ is given by
  \begin{equation}
  _{a}\mathcal{I}_{z}^{\beta}g(z)=\frac{1}{\Gamma(\beta)}\int_{a}^{z}
\frac{g(\tau)}{(S_{\mathbf{F}}^{\alpha}(z)-
S_{\mathbf{F}}^{\alpha}(\tau))^{1-\beta}}d_{\mathbf{F}}^{\alpha}\tau.
\end{equation}
\end{definition}
\begin{definition}
The right-sided Riemann-Liouville $F^{\alpha}$-derivative of order $\beta\in \mathbb{R}$ of $g$ is defined by
\begin{equation}
  _{z}\mathcal{D}_{b}^{\beta}g(z)=\frac{1}{\Gamma(n-\beta)}
(-D_{\mathbf{F}}^{\alpha})^{n}\int_{z}^{b}
\frac{g(\tau)}{(S_{\mathbf{F}}^{\alpha}(\tau)-
S_{\mathbf{F}}^{\alpha}(z))^{-n+\beta+1}}
d_{\mathbf{F}}^{\alpha}\tau,~~~g\in C^{n}(\mathbf{F}),
\end{equation}
and the left-sided Riemann-Liouville  $F^{\alpha}$-derivative of $g$  is defined by
\begin{equation}
  _{a}\mathcal{D}_{z}^{\beta}g(z)=\frac{1}{\Gamma(n-\beta)}
(D_{\mathbf{F}}^{\alpha})^{n}\int_{a}^{z}
\frac{g(\tau)}{(S_{\mathbf{F}}^{\alpha}(z)-
S_{\mathbf{F}}^{\alpha}(\tau))^{-n+\beta+1}}
d_{\mathbf{F}}^{\alpha}\tau,
\end{equation}
where  $n\alpha-\alpha\leq \beta< \alpha n$.
\end{definition}

\begin{definition}
The right-sided Caputo $F^{\alpha}$-derivative of order $\beta\in \mathbb{R}$ of $g$  is defined by
\begin{equation}
  _{z}^{C}\mathcal{D}_{b}^{\beta}g(z)=\frac{1}{\Gamma(n-\beta)}
 \int_{z}^{b}
(S_{\mathbf{F}}^{\alpha}(\tau)-S_{\mathbf{F}}^{\alpha}(z))^{n-\beta-1}
(-D_{\mathbf{F}}^{\alpha})^{n}g(\tau)d_{\mathbf{F}}^{\alpha}\tau,
\end{equation}
and the right-sided Caputo  $F^{\alpha}$-derivative of $g$  is defined by
\begin{equation}
  _{a}^{C}\mathcal{D}_{z}^{\beta}g(z)=\frac{1}{\Gamma(n-\beta)}
\int_{a}^{z}
(S_{\mathbf{F}}^{\alpha}(z)-S_{F}^{\alpha}(\tau))^{n-\beta-1}
(D_{\mathbf{F}}^{\alpha})^{n}g(\tau)d_{\mathbf{F}}^{\alpha}\tau,
\end{equation}
where  $n\alpha-\alpha\leq \beta< \alpha n$.
\end{definition}
\begin{table}
  \centering
  \begin{tabular}{|l|l|l|}
  \hline
 Ordinary calculus & Fractal calculus on sets& Fractal calculus on curves\\ \hline
 \parbox{3cm}{ $$\int_{0}^{z} \tau^{m}d\tau=\frac{1}{m+1}z^{m+1}$$} & \parbox{3cm}{$$\int_{0}^{z} (S_{\mathbf{F}}^{\alpha}(\tau))^{m}
d_{\mathbf{F}}^{\alpha}\tau=\frac{1}{m+1}(S_{\mathbf{F}}^{\alpha}(z))^{m+1}$$} & \parbox{3cm}{$$\int_{0}^{\theta} J_{\mathfrak{F}}^{\alpha}(\tau)^{m}d_{\mathfrak{F}}^{\alpha}\tau=
\frac{1}{m+1}(J(\theta))^{m+1}$$} \\ \hline
 \parbox{3cm}{ $$\frac{d}{dz}z^{m}=mz^{m-1}$$} & \parbox{3cm}{$$D_{\mathbf{F}}^{\alpha}
(S_{\mathbf{F}}^{\alpha}(z))^{m}=m(S_{\mathbf{F}}^{\alpha}(z))^{m-1} \chi_{\mathbf{F}}(z)$$}& \parbox{3cm}{$$D_{\mathfrak{F}}^{\alpha}(J(\theta))^{m}=m(J(\theta))^{m-1} \chi_{\mathfrak{F}}(\theta) $$}\\ \hline
  \parbox{3cm}{\begin{align}& _{0}I_{z}^{\beta}z^{m}\nonumber\\&=\frac{\Gamma(m+1)}
{\Gamma(m+\beta+1)}z^{m+\beta}
\nonumber\end{align}}&  \parbox{3cm}{\begin{align}& _{0}\mathcal{I}_{z}^{\beta}(S_{\mathbf{F}}^{\alpha}(z))^{m}\nonumber\\&
=\frac{\Gamma(m+1)}
{\Gamma(m+\beta+1)}(S_{\mathbf{F}}^{\alpha}(z))^{m+\beta}\nonumber\end{align}} & \parbox{3cm}{\begin{align}& _{0}\mathcal{I}_{\theta}^{\beta}J(\theta)^{m}\nonumber\\&=\frac{\Gamma(m+1)}
{\Gamma(m+\beta+1)}J(\theta)^{m+\beta}\nonumber\end{align}}\\ \hline
   \parbox{3cm}{\begin{align}&_{0}D_{z}^{\beta}z^{m}\nonumber\\&
=\frac{\Gamma(m+1)}{\Gamma(m-\beta+1)}z^{m-\beta}\nonumber\end{align}} & \parbox{3cm}{\begin{align}&_{0}\mathcal{D}_{z}^{\beta}
(S_{\mathbf{F}}^{\alpha}(z))^{m}\nonumber\\&=
\frac{\Gamma(m+1)}
{\Gamma(m-\beta+1)}(S_{\mathbf{F}}^{\alpha}(z))^{m-\beta}\chi_{\mathbf{F}}
(z)\nonumber\end{align}}  &\parbox{3cm}{\begin{align}&_{0}\mathcal{D}_{\theta}^{\beta}
J(\theta)^{m}\nonumber\\&
=\frac{\Gamma(m+1)}
{\Gamma(m-\beta+1)}(J(\theta))^{m-\beta}
\chi_{\mathfrak{F}}(\theta)\nonumber\end{align}}\\ \hline
\parbox{3cm}{\begin{align}&
  _{a}I_{z}^{\beta}(z-a)^{m}\nonumber\\&=
\frac{\Gamma(m+1)}{\Gamma(\beta+m+1)}(z-a)^{\beta+m}\nonumber\end{align}} & \parbox{3cm}{\begin{align}&
_{a}\mathcal{I}_{z}^{\beta}
(S_{\mathbf{F}}^{\alpha}(z)-S_{\mathbf{F}}^{\alpha}(a))^{m}\nonumber\\&=
\frac{\Gamma(m+1)}{\Gamma(\beta+m+1)}
(S_{\mathbf{F}}^{\alpha}(z)-
S_{\mathbf{F}}^{\alpha}(a))^{\beta+m}\nonumber\end{align}}  &
\parbox{3cm}{\begin{align}& _{a}\mathcal{I}_{\theta}^{\beta}
(J(\theta)-J(a))^{m}\nonumber\\&=\frac{\Gamma(m+1)}{\Gamma(\beta+m+1)}
(J(\theta)-J(a))^{\beta+n}\nonumber\end{align}} \\ \hline
 \parbox{3cm}{\begin{align}& _{a}D^{\beta}_{z}(z-a)^{m}=\nonumber\\&
\frac{\Gamma(m+1)}{\Gamma(m-\beta+1)}(z-a)^{m-\beta}
\nonumber\end{align}} & \parbox{3cm}{\begin{align}& _{a}\mathcal{D}^{\beta}_{z}(S_{\mathbf{F}}^{\alpha}
(z)-S_{\mathbf{F}}^{\alpha}(a))^{m}=\nonumber\\&
\frac{\Gamma(m+1)}{\Gamma(m-\beta+1)}(S_{\mathbf{F}}^{\alpha}(z)-
S_{\mathbf{F}}^{\alpha}(a))
^{m-\beta}
\nonumber\end{align}}&\parbox{3cm}{\begin{align}& _{a}\mathcal{D}^{\beta}_{\theta}(J(\theta)-J(a))^{m}=\nonumber\\&
\frac{\Gamma(m+1)}{\Gamma(m-\beta+1)}(J(\theta)-J(a))^{n-\beta}
\nonumber\end{align}}\\
  \hline\parbox{3cm}{\begin{align}_{-\infty}I_{z}^{\beta}\exp(\lambda z)=\lambda^{\beta}\exp(\lambda z)
\nonumber\end{align}} & \parbox{3cm}{\begin{align}_{-\infty}\mathcal{I}_{z}^{\beta}\exp(\lambda S_{\mathbf{F}}^{\alpha}(z))=\lambda^{\beta}\exp(\lambda S_{\mathbf{F}}^{\alpha}(z))
\nonumber\end{align}} &\parbox{3cm}{\begin{align}_{-\infty}\mathcal{I}_{\theta}^{\beta}\exp(\lambda J(\theta))=\lambda^{\beta}\exp(\lambda J(\theta))
\nonumber\end{align}}\\ \hline
\parbox{3cm}{\begin{align}_{-\infty}D_{z}^{\beta}\cos(z)=\cos
\left(z+\frac{\pi}{2}\beta\right)
\nonumber\end{align}} & \parbox{3cm}{\begin{align}_{-\infty}\mathcal{D}_{z}^{\beta}
\cos(S_{\mathbf{F}}^{\alpha}(z))=\cos\left(S_{\mathbf{F}}^{\alpha}(z)+
\frac{\pi}{2}\beta\right)
\nonumber\end{align}} &\parbox{3cm}{\begin{align}_{-\infty}\mathcal{D}_{\theta}^{\beta}
\cos(J(\theta))=\cos\left(J(\theta)+\frac{\pi}{2}\beta\right)
\nonumber\end{align}}\\ \hline
\parbox{3cm}{\begin{align}_{-\infty}D_{z}^{\beta}\exp(\lambda z)=\lambda^{\beta}\exp(\lambda z)
\nonumber\end{align}} & \parbox{3cm}{\begin{align}_{-\infty}\mathcal{D}_{z}^{\beta}\exp(\lambda S_{\mathbf{F}}^{\alpha}(z))=\lambda^{\beta}\exp(\lambda S_{\mathbf{F}}^{\alpha}(z))
\nonumber\end{align}} &\parbox{3cm}{\begin{align}_{-\infty}\mathcal{D}_{\theta}^{\beta}\exp(\lambda J(\theta))=\lambda^{\beta}\exp(\lambda J(\theta))
\nonumber\end{align}}\\ \hline

\end{tabular}
  \caption{Some analogies of fractal calculus and standard calculus.}\label{awsqzza}
\end{table}

\subsection{Non-local derivatives on fractal curves }

\begin{definition}\label{SSD1}
 The right-sided Riemann-Liouville  $F^{\alpha}$-integral of order $\beta\in \mathbb{R}$ of $f$ (on a fractal curve $F$) is defined by
\begin{equation}
  _{\theta}\mathcal{I}_{b}^{\beta}f(\theta)=\frac{1}{\Gamma(\beta)}\int_{D(\theta,b)}
\frac{f(\tau)}{(J(t)-J(\theta))^{1-\beta}}d_{\mathfrak{F}}^{\alpha}\tau,~~~\theta\in \mathfrak{F},
\end{equation}
and the left-sided Riemann-Liouville  $F^{\alpha}$-integral of $f$  is defined by
  \begin{equation}\label{bubad}
 _{a}\mathcal{I}_{\theta}^{\beta}f(\theta)=\frac{1}{\Gamma(\beta)}\int_{D(a,\theta)}
\frac{f(\tau)}{(J(\theta)-J(t))^{1-\beta}}d_{\mathfrak{F}}^{\alpha}\tau.
\end{equation}
\end{definition}
\begin{definition}
 The right-sided Riemann-Liouville $F^{\alpha}$-derivative of order $\beta\in \mathbb{R}$ of $f$ is defined by
\begin{equation}
  _{\theta}\mathcal{D}_{b}^{\beta}f(\theta)=\frac{1}{\Gamma(n-\beta)}
(-D_{\mathfrak{F}}^{\alpha})^{n}\int_{D(\theta,b)}
\frac{f(\tau)}{(J(\tau)-J(\theta))^{-n+\beta+1}}
d_{\mathfrak{F}}^{\alpha}\tau,~~~f \in C^{n}(\mathfrak{F}),
\end{equation}
and the left-sided Riemann-Liouville  $F^{\alpha}$-derivative of $f$ is defined by
\begin{equation}\label{sexy}
  _{a}\mathcal{D}_{\theta}^{\beta}f(\theta)=\frac{1}{\Gamma(n-\beta)}
(D_{\mathfrak{F}}^{\alpha})^{n}\int_{D(a,\theta)}
\frac{f(\tau)}{(J(\theta)-J(\tau))^{-n+\beta+1}}d_{\mathfrak{F}}^{\alpha}\tau,
\end{equation}
where  $n\alpha-\alpha\leq \beta< \alpha n$.
\end{definition}

\begin{definition}
 The right-sided Caputo $F^{\alpha}$-derivative of order $\beta\in \mathbb{R}$ of $f$ is defined by
\begin{equation}
  _{\theta}^{C}\mathcal{D}_{b}^{\beta}f(\theta)=\frac{1}{\Gamma(n-\beta)}
\int_{D(\theta,b)}
(J(\tau)-J(\theta))^{n-\beta-1}
(-D_{\mathfrak{F}}^{\alpha})^{n}f(\tau)d_{\mathfrak{F}}^{\alpha}\tau,
\end{equation}
and the right-sided Caputo  $F^{\alpha}$-derivative of order $\beta\in \mathbb{R}$ of $f$ is defined by
\begin{equation}
  _{a}^{C}\mathcal{D}_{\theta}^{\beta}f(\theta)=\frac{1}{\Gamma(n-\beta)}
\int_{D(a,\theta)}
(J(\theta)-J(\tau))^{n-\beta-1}
(D_{F}^{\alpha})^{n}f(\tau)d_{\mathfrak{F}}^{\alpha}\tau,
\end{equation}
where  $n\alpha-\alpha\leq \beta< \alpha n$.
\end{definition}
\begin{example}
Consider the following  function on a fractal curve
\begin{equation}
  f(\theta)=(J(\theta)-J(a))^{\nu}.
\end{equation}
Then its non-local derivative is
\begin{align}\label{iiiittr}
  _{a}\mathcal{I}_{\theta}^{\beta}(J(\theta)-J(a))^{\nu}&=
\frac{1}{\Gamma(\beta)}\int_{D(a,\theta)}
(J(t)-J(a))^{\nu}(J(\theta)-J(t))^{\beta-1}d_{\mathfrak{F}}^{\alpha}t.
\end{align}
By the change of variable $J(t)=J(a)+\zeta(J(\theta)-J(a))$, Eq.\eqref{iiiittr}, turns into
\begin{align}\label{iiii}
  _{a}\mathcal{I}_{\theta}^{\beta}(J(\theta)-J(a))^{\nu}&=
\frac{(J(\theta)-J(a))^{\beta+\nu}}{\Gamma(\beta)}\int_{D(0,1)}
(1-\zeta)^{\beta-1}\zeta^{\nu}d_{F}^{\alpha}\zeta\nonumber\\&=
\frac{B(\nu+1,\beta)}{\Gamma(\beta)}(J(\theta)-J(a))^{\beta+\nu}\nonumber\\&=
\frac{\Gamma(\nu+1)}{\Gamma(\beta+\nu+1)}(J(\theta)-J(a))^{\beta+\nu},
\end{align}
where $B(.,.)$ is the Beta function, which is defined by
\begin{equation}
  B(\nu+1,\beta)=\int_{D(0,1)}
(1-\zeta)^{\beta-1}\zeta^{\nu}d_{F}^{\alpha}\zeta=
\frac{\Gamma(\nu+1)\Gamma(\beta)}{\Gamma(\nu+\beta+1)}.
\end{equation}

\end{example}
\section{ Laplace transform of the fractal Riemann-Lioville differintegral \label{3g} }
In this section, we explore the Laplace transform of fractal Riemann-Lioville differintegral \cite{KamalAliKhaliliGolmankha}.
\begin{corollary}
 Laplace transform of the fractal Riemann-Lioville integral is defined by
\begin{equation}
  \mathcal{L}[_{0}I_{\theta}^{\beta}f(\theta)]=J(s)^{\beta}F(s),
\end{equation}
where
\begin{equation}
  F(s)=\mathcal{L}[f(\theta)]=\int_{D(0,\infty)}\exp(-J(s)J(\theta))
f(\theta)d_{\mathfrak{F}}^{\alpha}\theta.
\end{equation}
\end{corollary}

\begin{corollary}
The Laplace transform of the fractal Riemann-Liouville  derivative is defined by
\begin{equation}\label{rddddssa}
  \mathcal{L}[_{0}\mathcal{D}_{\theta}^{\beta}f(\theta)]=J(s)^{\beta}F(s)-\sum_{k=0}^{n-1}
J(s)^{k}[_{0}\mathcal{D}_{\theta}^{\beta-k-1}f(\theta)]_{\theta=0},
\end{equation}
where $n\alpha-\alpha\leq \beta<n\alpha$.
\end{corollary}
\begin{remark}
  We note that Eq.\eqref{rddddssa}, for the case $n=1$, becomes
\begin{equation}
  \mathcal{L}[_{0}\mathcal{D}_{\theta}^{\beta}f(\theta)]=J(s)^{\beta}F(s)-
~_{0}\mathcal{D}_{\theta}^{\beta-1}f(\theta)|_{\theta=0},~~~0\leq \beta <\alpha,
\end{equation}
and for $n=2$, it follows from Eq.\eqref{rddddssa} that
\begin{equation}
  \mathcal{L}[_{0}\mathcal{D}_{\theta}^{\beta}f(\theta)]=J(s)^{\beta}F(s)-
~_{0}\mathcal{D}_{\theta}^{\beta-1}f(\theta)|_{\theta=0}-J(s)
_{0}\mathcal{D}_{\theta}^{\beta-2}f(\theta)|_{\theta=0},~~~\alpha\leq \beta <2\alpha.
\end{equation}
\end{remark}
\begin{theorem}
The fractal Laplace transform of function $J(\theta)^{\eta m+\mu-1}E^{m}_{\eta,\mu}(aJ(\theta)^{\eta})$ is
\begin{equation}\label{rraqq85}
  \mathcal{L}[J(\theta)^{\eta m+\mu-1}E^{m}_{\eta,\mu}(aJ(\theta)^{\eta})]=
\frac{m!J(s)^{\eta-\mu}}{(J(s)^{\eta}-a)^{m+1}},
\end{equation}
  where
\begin{equation}\label{rcxsw9}
  E^{m}_{\eta,\mu}(aJ(\theta)^{\eta})=
\sum_{k=0}^{\infty}\frac{(k+m)!}{k!}\frac{a^{k}J(\theta)^{k\eta}}{\Gamma(\eta k+\eta m+ \mu)}.
\end{equation}
\end{theorem}
\begin{proof}
  Using the linearity of the fractal Laplace transform and Eq.\eqref{rcxsw9}, we can rewrite Eq.\eqref{rraqq85} as
\begin{equation}\label{iiiooopp}
  \mathcal{L}[J(\theta)^{\eta m+\mu-1}E^{m}_{\eta,\mu}(aJ(\theta)^{\eta})]=
\sum_{k=0}^{\infty}\frac{(k+m)!a^{k}}{k!\Gamma(\eta k+\eta m+ \mu)}\mathcal{L}[J(\theta)^{\eta k+\eta m+ \mu-1}].
\end{equation}
Applying the fractal Laplace transform of the power function $\mathcal{L}[J(\theta)^{\beta-1}]=\Gamma(\beta)J(s)^{-\beta}$,  Eq.\eqref{iiiooopp} turns into the following form
\begin{align}
  \mathcal{L}[J(\theta)^{\eta m+\mu-1}E^{m}_{\eta,\mu}(aJ(\theta)^{\eta})]&=
\sum_{k=0}^{\infty}\frac{(k+m)!a^{k}}{k!\Gamma(\eta k+\eta m+ \mu)}\frac{\Gamma(\eta k+\eta m+ \mu)}{J(s)^{\eta k+\eta m+ \mu}}\nonumber\\&=\sum_{k=0}^{\infty}\frac{(k+m)!}{k!}\frac{a^{k}}{J(s)^{\eta k+\eta m+ \mu}}\nonumber\\&=J(s)^{-\eta m-\mu}\sum_{k=0}^{\infty}
\frac{(k+m)!}{k!}\bigg(\frac{a}{J(s)^{\eta}}\bigg)^{k}\nonumber\\&=
J(s)^{-\eta m-\mu}D_{\mathfrak{F}}^{m}\sum_{k=m}^{\infty}\bigg(\frac{a}{J(s)^{\eta}}\bigg)^{k}
\nonumber\\&=J(s)^{-\eta m-\mu}D_{\mathfrak{F}}^{m}\frac{1}{(1-\frac{a}{J(s)^{\eta}})}\nonumber\\&=
J(s)^{-\eta m-\mu}\frac{m!}{(1-\frac{a}{J(s)^{\eta}})^{m+1}}\nonumber\\&=
\frac{m!J(s)^{\eta-\mu}}{(J(s)^{\eta}-a)^{m+1}},
\end{align}
which completes the proof.
\end{proof}

\begin{table}[H]
  \centering
  \begin{tabular}{|l|l|}
  \hline
 Function in $\theta$-space & Fractal  Laplace Transform $s$-space\\ \hline
  \parbox{4cm}{$$\frac{J(\theta)^{\beta-1}}{\Gamma(\beta)}$$} & \parbox{3cm}{$$\frac{1}{J(s)^{\beta}}$$} \\ \hline
\parbox{4cm}{$$J(\theta)^{\beta-1}E_{\beta,\beta}(aJ(\theta)^{\beta})$$} & \parbox{3cm}{$$\frac{1}{J(s)^{\beta}-a}$$} \\ \hline
\parbox{4cm}{$$E_{\beta}(-aJ(\theta)^{\beta})$$} & \parbox{3cm}{$$\frac{J(s)^{\beta}}{J(s)(J(s)^{\beta}+a)}$$} \\ \hline
\parbox{4cm}{$$1-E_{\beta}(-aJ(\theta)^{\beta})$$} & \parbox{3cm}{$$\frac{a}{J(s)(J(s)^{\beta}+a)}$$} \\ \hline
\parbox{4cm}{$$J(\theta)^{\beta}E_{1,\beta+1}(aJ(\theta))$$} & \parbox{3cm}{$$\frac{1}{J(s)^{\beta}(J(s)-a)}$$} \\ \hline
\parbox{4cm}{$$J(\theta)^{\beta-1}E_{\nu,\beta}(aJ(\theta)^{\nu})$$} & \parbox{3cm}{$$\frac{J(s)^{\nu-\beta}}{(J(s)^{\nu}-a)}$$} \\ \hline
\parbox{4cm}{$$J(\theta)^{\eta m+\mu-1}E^{m}_{\eta,\mu}(aJ(\theta)^{\eta})$$} & \parbox{3cm}{$$\frac{m!J(s)^{\eta-\mu}}{(J(s)^{\eta}-a)^{m+1}}$$} \\ \hline
\end{tabular}
  \caption{Some formulas of the fractal Laplace transforms.}\label{awsqzza2}
\end{table}
\begin{corollary}
The Laplace transform of the Caputo derivative is defined by
 \begin{equation}
   \mathcal{L}[_{0}^{C}\mathcal{D}_{\theta}^{\beta}f(\theta)]=J(s)^{\beta}F(s)-
\sum_{k=0}^{n-1}
J(s)^{\beta-k-1}[(D_{\mathfrak{F}}^{\alpha})^{k}f(\theta)]_{\theta=0},
 \end{equation}
where $n\alpha-\alpha<\beta\leq n\alpha$.
\end{corollary}

\begin{example}
Consider the non-local fractal differential equation
\begin{equation}\label{xxxxx}
  _{0}D_{\theta}^{\beta}y(\theta)-\lambda y(\theta)=h(\theta),~~(\theta>0),
\end{equation}
with the non-local initial condition
\begin{equation}
  [_{0}D_{\theta}^{\beta-k}y(\theta)]_{\theta=0}=c_{k},~~~(k=1,2,...,n),
\end{equation}
where $n\alpha-\alpha\leq\beta<n\alpha$ and $\lambda$ is constant.\\ To solve this problem, taking the fractal non-local Laplace transform of Eq.\eqref{xxxxx} and using Eq.\eqref{rddddssa}  we obtain
\begin{equation}
  J(s)^{\beta}Y(J(s))-\lambda Y(J(s))=H(J(s))+\sum_{k=1}^{n}c_{k}J(s)^{k-1}.
\end{equation}
Solving for $Y(J(s))$, we have
\begin{equation}
  Y(J(s))=\frac{H(J(s))}{J(s)^{\beta}-\lambda}+\sum_{k=1}^{n}c_{k}
\frac{J(s)^{k-1}}{J(s)^{\beta}-\lambda}.
\end{equation}
The fractal inverse Laplace transform and Eq.\eqref{rraqq85} give the solution
\begin{equation}
  y(\theta)=\sum_{k=1}^{n}c_{k}J(\theta)^{\beta-k}E_{\beta,\beta-k+1}
(\lambda J(\theta)^{\beta})+\int_{D(0,\theta)}
(J(\theta)-J(t))^{\beta-1}
E_{\beta,\beta}(\lambda(J(\theta)-J(t))^{\beta})h(t)d_{\mathfrak{F}}^{\alpha}t.
\end{equation}
\end{example}

\section{The analogue of Mellin integral transform on fractal calculus \label{40g}}
In this section, we generalize the Mellin integral transform to include a function with fractal support.
\begin{definition}\label{Wqaww741}
The Mellin integral transform $M(s)$ of a function $f$ on fractals is defined by
\begin{equation}
  M(s)=\int_{D(0,\infty)}f(\theta)J(\theta)^{J(s)-1}
d_{\mathfrak{F}}^{\alpha}\theta,
\end{equation}
where $J(s)$ is complex, such as $\lambda_{1}<Re(J(s))<\lambda_{2}$. The Mellin transform exists if $f$ is fractal piecewise continuous in every closed interval $D(a,b)\subset D(0,\infty)$ and
\begin{equation}
  \int_{D(0,1)}|f(\theta)|J(\theta)^{\lambda_{1}-1}d_{\mathfrak{F}}^{\alpha}\theta<\infty,~~~
\int_{D(1,\infty)}|f(\theta)|J(\theta)^{\lambda_{2}-1}d_{\mathfrak{F}}^{\alpha}\theta<\infty.
\end{equation}
\end{definition}

\begin{definition}
The fractal inverse Mellin transform formula is defined by
\begin{equation}
  f(\theta)=\frac{1}{2\pi i}\int_{\lambda-i\infty}^{\lambda+i\infty}M(z)J(\theta)
^{-z}dz,~~~(0<J(\theta)<\infty),
\end{equation}
where $\lambda_{1}<\lambda<\lambda_{2}$.
\end{definition}
\begin{example}
  The fractal Mellin transform of function $J(\theta)^{\nu}f(\theta)$ is
\begin{equation}
  \mathcal{M}(J(\theta)^{\nu}f(\theta))=\int_{D(0,\infty)}J(\theta)^{\nu}
f(\theta)J(\theta)^{J(s)-1}d_{\mathfrak{F}}^{\alpha}\theta=M(\nu+J(s)).
\end{equation}
\end{example}
\begin{definition}
The  fractal Mellin convolution is defined by
\begin{equation}
  f(\theta)*h(\theta)=
\int_{D(0,\infty)}f(J(\theta)\tau)h(\tau)d_{\mathfrak{F}}^{\alpha}\tau,
\end{equation}
and its fractal Mellin transform is
\begin{equation}
  \mathcal{M}\{f(J(\theta)\tau)g(\tau)
d_{\mathfrak{F}}^{\alpha}\tau\}=M_{1}(J(s))M_{2}(1-J(s)),
\end{equation}
where $M_{1}(.)$ and $M_{2}(.)$ are the fractal Mellin transform of $f(\theta)$ and $h(\theta)$, respectively.
\end{definition}
\begin{corollary}\label{Sic}
  The fractal Mellin transform holds
\begin{equation}
  \mathcal{M}\bigg[J(\theta)^{\kappa}
\int_{D(0,\infty)}J(t)^{\tau}f(J(\theta)t)
g(t)d_{\mathfrak{F}}^{\alpha}t\bigg]=
M_{1}(J(s)+\kappa)M_{2}(1-J(s)-\kappa+\tau).
\end{equation}
\end{corollary}
\begin{theorem}\label{Kerr}
  The fractal Mellin transform of $n$-th $F^{\alpha}$-derivative is
\begin{equation}
  \mathcal{M}[(D_{\mathfrak{F}}^{\alpha})^{n}f(\theta)]=
\frac{\Gamma(1-J(s)+n)}{\Gamma((1-J(s))}M(J(s)-n).
\end{equation}
\end{theorem}
 \begin{proof}
To prove of this theorem, we use the fractal integral by part as follows
   \begin{align}\label{uuiiuu}
  \mathcal{M}[(D_{\mathfrak{F}}^{\alpha})^{n}f(\theta)]&=\int_{D(0,\infty)} (D_{\mathfrak{F}}^{\alpha})^{n}f(\theta)J(\theta)^{J(s)-1}
d_{\mathfrak{F}}^{\alpha}\theta\nonumber\\&=
[(D_{\mathfrak{F}}^{\alpha})^{n-1}f(\theta)J(\theta)
^{J(s)-1}\bigg|_{0}^{\infty}-(J(s)-1)\int_{D(0,\infty)} D_{\mathfrak{F}}^{\alpha})^{n-1}f(\theta)J(\theta)^{J(s)-2}
d_{\mathfrak{F}}^{\alpha}\theta\nonumber\\&=
[(D_{\mathfrak{F}}^{\alpha})^{n-1}f(\theta)J(\theta)
^{J(s)-1}\bigg|_{0}^{\infty}-(J(s)-1)
\mathcal{M}[(D_{\mathfrak{F}}^{\alpha})^{n-1}f(\theta)]\nonumber\\&=...
\nonumber\\&=\sum_{k=0}^{n-1}(-1)^{k}\frac{\Gamma(J(s))}{\Gamma(J(s)-k)}[
D_{\mathfrak{F}}^{\alpha})^{n-k-1}f(\theta)J(\theta)^{J(s)-k-1}
\bigg|_{0}^{\infty}+
(-1)^{n}\frac{\Gamma(J(s))}{\Gamma(J(s)-k)}M(J(s)-n)\nonumber\\&=
\sum_{k=0}^{n-1}\frac{\Gamma(1-J(s)+k)}{\Gamma(1-J(s))}[
D_{\mathfrak{F}}^{\alpha})^{n-k-1}f(\theta)J(\theta)^{J(s)-k-1}
\bigg|_{0}^{\infty}+
\frac{\Gamma(1-J(s)+n)}{\Gamma(1-J(s))}M(J(s)-n).
\end{align}
The first term of the right side of Eq.\eqref{uuiiuu} as the limits $t=0$ and $t=\infty$  is zero. Then Eq.\eqref{uuiiuu} takes the form
\begin{equation}
  \mathcal{M}[(D_{\mathfrak{F}}^{\alpha})^{n}f(\theta)]=
\frac{\Gamma(1-J(s)+n)}{\Gamma(1-J(s))}M(J(s)-n),
\end{equation}
which completes the proof.
 \end{proof}
\begin{theorem}
  The Mellin transform of the fractal Riemann-Liouville is
\begin{equation}
  \mathcal{M}[_{0}\mathcal{I}_{\theta}^{\beta}f(\theta)]=
\frac{\Gamma(1-J(s)-\beta)}{\Gamma(1-J(s))}M(J(s)+\beta),
\end{equation}
where $M(J(s))$ is the Mellin transform of the function $f(J(\theta))$.
\end{theorem}
\begin{proof}
  To prove this theorem, recalling the Definition \ref{bubad} and setting $a=0$ we have
\begin{align}\label{tttffggdd}
  _{0}\mathcal{I}_{\theta}^{\beta}f(\theta)&=
\frac{1}{\Gamma(\beta)}\int_{D(0,\theta)}(J(\theta)-J(t))^{\beta-1}f(t)
d_{\mathfrak{F}}^{\alpha}t,\nonumber\\
\intertext{by setting  $J(t)=J(\theta)\zeta$ we have}
_{0}\mathcal{I}_{\theta}^{\beta}f(\theta)&=
\frac{J(\theta)^{\beta}}{\Gamma(\beta)}
\int_{D(0,1)}(1-\zeta)^{\beta-1}f(J(\theta)\zeta)d_{\mathfrak{F}}^{\alpha}\zeta~~~
\nonumber\\&=\frac{J(\theta)^{\beta}}{\Gamma(\beta)}
\int_{D(0,\infty)}f(J(\theta)\zeta)h(\zeta)d_{\mathfrak{F}}^{\alpha}\zeta,
\end{align}
where
\begin{equation}
  h(\theta)=\left\{
              \begin{array}{ll}
                (1-\theta)^{\beta-1}, & (0\leq \theta <1); \\
                0, & (\theta\geq1).
              \end{array}
            \right.
\end{equation}
The fractal Mellin transform of $h(\theta)$ is
\begin{equation}\label{cxzaqweeees}
  \mathcal{M}[h(\theta)]=B(\beta,J(s))=
\frac{\Gamma(\beta)\Gamma(J(s))}{\Gamma(\beta+J(s))},
\end{equation}
where $B(.,.)$ is the Beta function. Using Corollary \ref{Sic}, Eq.\eqref{tttffggdd} and Eq.\eqref{cxzaqweeees}, we have

\begin{equation}
  \mathcal{M} [_{0}\mathcal{I}_{\theta}^{\beta}f(\theta)]=
\frac{1}{\Gamma(\beta)}M(J(s)+\beta)B(\beta,1-J(s)-\beta),
\end{equation}
 or
\begin{equation}
  \mathcal{M} [_{0}\mathcal{I}_{\theta}^{\beta}f(\theta)]=
\frac{\Gamma(1-J(s)-\beta)}{\Gamma(1-J(s))}M(J(s)+\beta),
\end{equation}
which completes the proof.
\end{proof}
\begin{theorem}
  The fractal Mellin transform of the Riemann-Liouville derivative is
\begin{equation}
  \mathcal{M}[_{0}\mathcal{D}_{\theta}^{\beta}f(\theta)]=
\frac{\Gamma(1-J(s)-\beta)}{\Gamma(1-J(s))}M(J(s)-\beta).
\end{equation}
\end{theorem}
\begin{proof}
By view of  Definitions \ref{sexy} and \ref{Wqaww741}, we have
\begin{align}
  \mathcal{M}[_{0}\mathcal{D}_{\theta}^{\beta}f(\theta)]&=
\mathcal{M}[(D_{F}^{\alpha})^{n}~_{0}\mathcal{D}_{\theta}^{\beta-n}f(\theta)],
\end{align}
and setting  $h(\theta)=~_{0}\mathcal{D}_{\theta}^{-(n-\beta)}f(\theta)$ we have
\begin{align}
\mathcal{M}[_{0}\mathcal{D}_{\theta}^{\beta}f(\theta)]=\mathcal{M}
[(D_{F}^{\alpha})^{n}h(\theta)].
\end{align}
Recalling Theorem \ref{Kerr}, we can write
\begin{align}\label{uuhygtf}
\mathcal{M}[(D_{\mathfrak{F}}^{\alpha})^{n}h(\theta)]&=\sum_{k=0}^{n-1}\frac{\Gamma(1-J(s)+k)}{\Gamma(1-J(s))}[
D_{\mathfrak{F}}^{\alpha})^{n-k-1}h(\theta)J(\theta)^{J(s)-k-1}\bigg|_{0}^{\infty}+
\frac{\Gamma(1-J(s)+n)}{\Gamma(1-J(s))}H(J(s)-n),
\end{align}
where $H(J(s))$ is the fractal Mellin transform of $h(J(\theta))$. By replacing $h(\theta)$,  we rewrite Eq.\eqref{uuhygtf} as
\begin{align}
\mathcal{M}[_{0}\mathcal{D}_{\theta}^{\beta}f(\theta)]&=\sum_{k=0}^{n-1}
\frac{\Gamma(1-J(s)+k)}{\Gamma(1-J(s))}[
D_{\mathfrak{F}}^{\alpha})^{n-k-1}~
_{0}\mathcal{D}_{\theta}^{-(n-\beta)}f(\theta)J(\theta)^{J(s)-k-1}
\bigg|_{0}^{\infty}\nonumber\\&+
\frac{\Gamma(1-J(s)+n)}{\Gamma(1-J(s))}\frac{\Gamma(1-(J(s)-n)-(n-\beta))}
{\Gamma(1-(J(s)-n))}M((J(s)-n)+(n-\beta)),
\end{align}
or
\begin{align}\label{tt566}
\mathcal{M}[_{0}\mathcal{D}_{\theta}^{\beta}f(\theta)]&=\sum_{k=0}^{n-1}
\frac{\Gamma(1-J(s)+k)}{\Gamma(1-J(s))}[
_{0}\mathcal{D}_{\theta}^{\beta-k-1}f(\theta)J(\theta)^{J(s)-k-1}
\bigg]_{0}^{\infty}+
\frac{\Gamma(1-J(s)+\beta)}{\Gamma(1-J(s))}M(J(s)-\beta).
\end{align}
In the case of $1<\beta<\alpha$,  Eq.\eqref{tt566} turns into
\begin{align}\label{uuuhh}
\mathcal{M}[_{0}\mathcal{D}_{\theta}^{\beta}f(\theta)]&=~
_{0}\mathcal{D}_{\theta}^{\beta-1}f(\theta)J(\theta)^{J(s)-1}
\bigg]_{0}^{\infty}+
\frac{\Gamma(1-J(s)+\beta)}{\Gamma(1-J(s))}M(J(s)-\beta).
\end{align}
By substituting the limits $\theta=0$ and $\theta=\infty$, the first term of right side of Eq.\eqref{uuuhh} is zero, and  Eq.\eqref{uuuhh} on the simplest becomes
\begin{equation}
  \mathcal{M}[_{0}\mathcal{D}_{\theta}^{\beta}f(\theta)]=
\frac{\Gamma(1-J(s)+\beta)}{\Gamma(1-J(s))}M(J(s)-\beta).
\end{equation}
\end{proof}
\begin{theorem}
The fractal Mellin transform of the Caputo derivative is
  \begin{equation}
  \mathcal{M}[_{0}^{C}\mathcal{D}_{\theta}^{\beta}f(\theta)]=
\frac{\Gamma(1-J(s)+\beta)}{\Gamma(1-J(s))}M(J(s)-\beta).
\end{equation}
\end{theorem}
\begin{proof}
  For proving this theorem, by virtue of Definition \ref{Wqaww741} we have
\begin{align}
  \mathcal{M}[_{0}^{C}\mathcal{D}_{\theta}^{\beta}f(\theta)]=
\mathcal{M}[_{0}\mathcal{D}_{\theta}^{-(n-\beta)}(D_{\mathfrak{F}}^{\alpha})^{n}f(\theta)].
\end{align}
 By supposing $h(\theta)=(D_{\mathfrak{F}}^{\alpha})^{n}f(\theta)$ we get
\begin{align}
  \mathcal{M}[_{0}^{C}\mathcal{D}_{\theta}^{\beta}f(\theta)]&=
\mathcal{M}[_{0}\mathcal{D}_{\theta}^{-(n-\beta)}h(\theta)]\nonumber\\&=
\frac{1-J(s)-(n-\beta)}{\Gamma(1-J(s))}H(J(s)+(n+\beta))\nonumber\\&=
\frac{\Gamma(1-J(s)-n+\beta))}{\Gamma(1-J(s))}\bigg\{
\sum_{k=0}^{n-1}\frac{\Gamma(1-(J(s)+n-\beta)+k)}{\Gamma(1-(J(s)+n-\beta))}
\bigg[(D_{\mathfrak{F}}^{\alpha})^{n-k-1}f(\theta)J(\theta)^
{(J(s)+n-\beta)-k-1}\bigg]_{0}^{\infty}\nonumber\\&
+\frac{\Gamma(1-(J(s)+n-\beta)+n)}{\Gamma(J(s)+n-\beta-n)}
M((J(s)+n-\beta)-n)\bigg\}\nonumber\\&=
\sum_{k=0}^{n-1}\frac{\Gamma(1-J(s)-n+\beta+k)}{\Gamma(1-J(s))}
\bigg[(D_{\mathfrak{F}}^{\alpha})^{n-k-1}f(\theta)J(\theta)^
{J(s)+n-\beta-k-1}\bigg]_{0}^{\infty}\nonumber\\&
+\frac{\Gamma(1-J(s)-\beta)}{\Gamma(1-J(s))}
M(J(s)-\beta),
\end{align}
or
\begin{align}\label{yygtfrdes}
  \mathcal{M}[_{0}^{C}\mathcal{D}_{\theta}^{\beta}f(\theta)]&=
\sum_{k=0}^{n-1}\frac{\Gamma(\beta-n-J(s))}{\Gamma(1-J(s))}
\bigg[(D_{\mathfrak{F}}^{\alpha})^{k}f(\theta)J(\theta)^
{J(s)-\beta+k}\bigg]_{0}^{\infty}\nonumber\\&
+\frac{\Gamma(1-J(s)-\beta)}{\Gamma(1-J(s))}
M(J(s)-\beta).
\end{align}
For the case of $1<\beta<\alpha$, Eq.\eqref{yygtfrdes} turn into
\begin{equation}
  \mathcal{M}[_{0}^{C}\mathcal{D}_{\theta}^{\beta}f(\theta)]=
\frac{\Gamma(\beta-J(s))}{\Gamma(1-J(s))}
\bigg[f(\theta)J(\theta)^{J(s)-\beta}\bigg]_{0}^{\infty}+
\frac{\Gamma(1-J(s)-\beta)}{\Gamma(1-J(s))}
M(J(s)-\beta).
\end{equation}
By taking into account of the limits $\theta=0$ and $\theta=\infty$, we arrive at
\begin{equation}
  \mathcal{M}[_{0}^{C}\mathcal{D}_{\theta}^{\beta}f(\theta)]=
\frac{\Gamma(1-J(s)-\beta)}{\Gamma(1-J(s))}
M(J(s)-\beta).
\end{equation}
Then the proof is complete.
\end{proof}

\begin{example}
  Consider the fractal differential equation
\begin{equation}\label{uytre78}
  D_{\mathfrak{F}}^{\alpha}f(\theta)+f(\theta)=0,
\end{equation}
By taking the fractal Mellin transform we have
\begin{equation}\label{Tt7887rqazs}
  -(J(s)-1)M(J(s)-1)+M(J(s))=0~~~\textmd{or}~~~M(J(s)-1)=J(s)M(J(s)).
\end{equation}
The solution of Eq.\eqref{Tt7887rqazs} is $M(J(s))=\Gamma(J(s))$. Then the fractal inverse Mellin transform gives the solution of Eq.\eqref{uytre78} by
\begin{align}
  f(\theta)&=\frac{1}{2\pi i}\int_{c-i\infty}^{c+i\infty}J(\theta)^{-z}\Gamma(z)d_{\mathfrak{F}}^{\alpha}z\nonumber\\&=
\sum_{n=0}^{\infty}Res_{z=-n}J(s)^{-z}\Gamma(z)\nonumber\\&=\sum_{n=0}^{\infty}J(s)^{n}
\frac{(-1)^{n}}{n!}=\exp(-J(\theta)).
\end{align}
\end{example}

\section{Conclusion \label{5g}}
In this paper,  the Riemann-Liouville integrals and derivatives and Caputo derivatives of fractal sets and curves have been defined. The processes on fractal spaces with memory effect can be modeled by them. Fractal differential equations have been solved by using the Mellin transform. The solution of the fractal differential equation is not differentiable in the standard calculus. The results become standard cases if  we choose $\alpha=1$.

\section{Appendix}
In this section, we give some formulas  used in the paper \cite{trifce2020fractional,podlubny1998fractional}.\\
The Mittag-Leffler one-parameter function is defined by
\begin{equation}\label{MMer1}
  E_{\alpha}(x)=\sum_{k=0}^{\infty}\frac{x^{k}}{\Gamma(\alpha k+1)},~~(\alpha>0,~~\beta>0).
\end{equation}
The Mittag-Leffler two-parameter function is defined by
\begin{equation}\label{MMer12}
  E_{\alpha,\beta}(x)=\sum_{k=0}^{\infty}\frac{x^{k}}{\Gamma(\alpha k+\beta)},~~(\alpha>0,~~\beta>0),
\end{equation}
where  $E_{\alpha,1}(x)= E_{\alpha}(x)$.\\
The gamma function is defined by
\begin{equation}
  \Gamma(x)=\int_{0}^{\infty}e^{-\tau}\tau^{x-1}d\tau.
\end{equation}
The beta function is defined by
\begin{equation}
  \mathbf{B}(x,\nu)=\int_{0}^{1}\tau^{x-1}(1-\tau)^{\nu-1}d\tau,~~~ (Re(x)>0,~~Re(\nu)>0).
\end{equation}
The beta function relates to the gamma function by the following equation
\begin{equation}
  \mathbf{B}(x,\nu)=\frac{\Gamma(x)\Gamma(\nu)}{\Gamma(x+\nu)}.
\end{equation}

\textbf{Acknowledgements:} Cristina Serpa acknowledges partial funding by national funds through FCT-Foundation for Science and Technology, project reference: UIDB/04561/2020.
\bibliographystyle{elsarticle-num}
\bibliography{Mellinnn}

\end{document}